\documentclass[a4paper,11pt]{article}

\input{styleart.sty}
\usepackage{amsmath}
\usepackage{verbatim}
\usepackage{cite}
\title{Some Lower Bounds on Shelah Rank in the Free Group}
\date{\today}
\author{Javier de la Nuez Gonz\'alez\footnote{The research leading to these results has received funding from the European Research Council under the European Unions Seventh Framework Programme (FP7/2007- 2013)/ERC Grant Agreement No. 291111}
, Chlo\'e Perin\footnote{This research was partially supported by the ISRAEL SCIENCE FOUNDATION (grant No. 1480/16)}\; and Rizos Sklinos}

\usepackage{enumerate}

\begin{document}

\newcommand{\SR}[0]{R^{\infty}}
\newcommand{\setof}[2]{\{\, #1\,\,|\,\,#2\,\}}
\newcommand{\bunion}[3]{\bigcup\limits_{\substack{#1}}^{#3}#2}
\newcommand{\enum}[2]{\begin{enumerate}[\hspace*{0.5cm}#1]#2\end{enumerate}}

\newcommand{\nil}{\emptyset}
\newcommand{\subg}[1]{\langle #1 \rangle}
\newcommand{\frp}{\ast}
\newcommand{\submod}[0]{\preccurlyeq}
\newcommand{\GG}{\ensuremath{\mathcal{G}}}

\newtheorem{clm}[theorem]{Claim}
\maketitle

\begin{abstract}
	We give some lower bounds on the Shelah rank of varieties in the free group whose coordinate groups are hyperbolic towers.
\end{abstract}

\section{Introduction}

  We initiate the study of Shelah rank for definable sets in the theory of the free group.
  We are motivated by the recent addition of the above mentioned theory to the family of stable theories by Sela (see \cite{SelaStability}).
 Poizat had showed in \cite{PoizatGenericAndRegular} that the theory of free groups of infinite countable rank is not superstable, and the solution to Tarski problem by Sela \cite{Sela4} (see also \cite{KharlampovichMyasnikov}) implies that this theory coincide with that of any non abelian finitely generated free groups. 
 
  Still one would like to recover the "superstable part" of the free group, i.e. formulas that admit ordinal Shelah rank, as well as to calculate said rank.
  
  In this paper we are able to give some lower bounds on Shelah rank for sets defined by systems of equations of a particular type, namely varieties whose "co-ordinate group" has the structure of a hyperbolic tower (see Section \ref{Towers} for a definition).
  One can read off a lower bound for the Shelah rank of such "hyperbolic tower varieties" from the complexity of the tower under consideration.
  
  Tower varieties are of special interest in the study of definable sets in the free groups. Indeed, Sela proved in \cite{SelaImaginaries} that to any definable set $X$ can be associated a finite set of tower varieties, that he calls an \textbf{envelope} of $X$, so that points which he calls \textbf{generic} with respect to this envelope must lie in the definable set. There are reasons to hope that the Shelah rank of the definable set is closely related to that of the tower varieties in an envelope of it, which is why we started our study of Shelah rank by trying to compute the rank of such varieties.
  
  Our proof makes heavy use of the curve complex (see Section \ref{FreeProductSec}),
  which is not surprising since the curve complex has already been proved useful in describing forking independence
  in the theory of the free group (see \cite{PerinSklinosForking}).

\section{Shelah Rank}
  
  \subsection{Definition and first properties}
    
    We fix a stable first order theory $T$ and we work in a big saturated model $\mathbb{M}$. This enables us to assume that any countable model $M$ of the theory is elementary embedded into $\mathbb{M}$, in such a way that any automorphism of $M$ extends to an automorphism of $\mathbb{M}$.
    
    \begin{definition}
    	Let $\phi(\bar{x},\bar{b})$ be a first order formula over parameters $\bar{b}$ living in $\mathbb{M}$ and let $A$ be a subset of $\mathbb{M}$. Then
    	$\phi(\bar{x},\bar{b})$ forks over $A$ if there are $k<\omega$ and an infinite sequence $(\bar{b}_i)_{i<\omega}$ of tuples in $\mathbb{M}$ with
    	$tp(\bar{b}_i/A)=tp(\bar{b}/A)$ for each $i<\omega$, such that the set $\{\phi(\bar{x},\bar{b}_i): i<\omega\}$ is $k$-inconsistent.
    \end{definition}
    
    Recall that a definable set in a theory $T$ is an equivalence class of formulas with parameters under the equivalence relation generated by $\phi(x,b) \sim \psi(x,c)$ iff $T \models \phi(x,b) \leftrightarrow \psi(x,c)$. For any model $M$ of $T$ containing $b$ we then denote $\phi(M,b)$ the subset $ \{ m \mid M \models \phi(m,b)\}$ of $M^{\abs{x}}$, and we call it a definable subset of $M$. 
    
    \begin{remark}  An equivalent characterization of forking can be given purely in terms of the set $X$ defined by $\phi$ in $\mathbb{M}$ and the action of the automorphism group of $\mathbb{M}$: the formula $\phi$ forks over $A$ if and only if there are $k < \omega$ and an infinite sequence of automorphisms $\sigma_n \in \Aut_A(\mathbb{M})$ such that the sets $\sigma_n(X)$ are $k$-wise disjoint.
    \end{remark}

In particular, this implies that forking over $A$ is really a property of the definable set, not just of the formula $\phi(x,b)$, justifying the following definition:
\begin{definition} We say that a definable set $X$ forks over $A$ if some (equivalently any) formula defining it forks over $A$. 
\end{definition}

    \begin{definition}
    	Given two non-empty definable sets $X,Y$ we say that $X<Y$ if $X\subset Y$ and $Y$ is definable over a set $A$ of parameters such that $X$ forks over $A$.
    \end{definition}
    
    Note that this defines a partial order on the class of definable sets (alternatively, of definable formulas).
    
    \begin{definition}
    	We say that a definable set (alternatively, any of its defining formulas) is superstable if there is no infinite descending chain of definable sets $X:=X_1>X_2>\ldots>X_n>\ldots$ .
    \end{definition}

    \begin{definition} We define the \textbf{Shelah rank relative to a set} $A$ of parameters as the lowest function which assigns a value $\SR_A(X)\in Ord\cup\{\infty\}$ (where $\infty$ is an imaginary value bigger than any ordinal) to every non-empty definable set $X$, and satisfies that if $X \subseteq Y$ for some definable set $Y$, and there exists a set of parameters $B$ containing $A$ over which $Y$ is definable and such that $X$ forks over $B$, then $\SR_A(X) < \SR_A(Y)$.
    
    	This is equivalent to defining the Shelah rank in the theory where the language has been enriched by some constants for the elements of $A$. 
    	
    	Note that if $A \subseteq A'$ we have $\SR_{A'}(X) \leq \SR_A(X)$. One refers to $\SR_{\emptyset}(X)$ as simply the Shelah rank of $X$ or $\SR(X)$.
    \end{definition}
    
    In particular, $\SR(X)$ is the foundation rank of the partial order $<$ and a definable set $X$ is superstable if and only if $\SR(X)\in Ord$. 
    
    \begin{fact}
    	\label{rank basic facts}
    	\begin{enumerate}
    		 \item If there is a definable bijection between the definable sets $X$ and $Y$, then $\SR(X)=\SR(Y)$.
    		 \item $\SR_{A}(X)=\SR_{A}(\sigma(X))$ for $X$ any set definable over a model $M$ containing $A$ and any automorphism $\sigma$ of $M$ fixing $A$.
    		 \item If $X$ and $Y$ are definable sets with $X \subseteq Y$, then $\SR_{A}(X) \leq \SR_{A}(Y)$ for any $A$.
    	\end{enumerate}
    \end{fact}

  \subsection{Lower bound on Shelah rank through automorphisms}
    
    By passing to the expansion $\mathbb{M}^{eq}$ one can always assume a given first order theory has elimination of imaginaries. This implies the existence, for every definable set $X$, of a defining formula $\phi(x,a)$ for $X$ such that any automorphism of (an extension of) $\mathbb{M}$ that fixes the definable set $X$ set-wise must fix the parameter $a\in\mathbb{M}^{<\omega}$ as well. This allows us to formulate the criterion for an increase in Shelah rank that we will use in this paper:
    
    \begin{lemma}
    \label{automorphisms and rank}Let $\mathbb{M}$ be a model with a fixed parameter subset $A$. Suppose  $X \subset M^{d}$ is definable (not necessarily over $A$ only).
    If there exists a subset $Y$ of $X$ definable over $\mathbb{M}$, an integer $k$, and a sequence $\sigma_n \in \Aut_A(\mathbb{M})$ such that the translates $\sigma_n(Y)$ are $k$-wise disjoint, and the $\sigma_n$ preserve $X$ set-wise, then $R_A^{\infty}(X) \geq R_A^{\infty}(Y) + 1$.
    \end{lemma}
    \begin{proof}
    Indeed, by the discussion above we can assume that the $\sigma_{n}$ fix a defining tuple of parameters $b$ for $X$. If $\psi(x,c)$ is a defining formula for $Y$ then the sequence $\{\sigma_n(c)\}$ witnesses the fact that $\psi$ forks over $Ab$, so that  $R_A^{\infty}(\phi) > R_A^{\infty}(\psi)$.

    \end{proof}

  \subsection{Shelah rank of varieties over the free group}
    
    In the particular case of the free group, \cite{PerinSklinosSuperstable} gives a sufficient criterion to determine whether a formula has infinite Shelah rank. We use it here to give a sufficient condition for a variety to have infinite Shelah rank.
    \newcommand{\genr}[1]{\pmb{#1}}   
    \newcommand{\xx}[0]{\pmb{x}}
    
    Let $a = (a_1, \ldots, a_n)$ be a basis of the free group $\F$. To any finite system of equations $\Sigma(x,a)=1$ in the language of groups with parameters in $\F$ one can associate the group $G_{\Sigma}$ given by the presentation $\subg{\genr{x},a\,|\,\Sigma(\genr{x},a)=1}$, where $\xx$ is a tuple of generators in bijective correspondence with the variables in the tuple $x$. 
    
    The set $V_{\Sigma}\subset\F^{|x|}$ of solutions to the system $\Sigma(x,a)=1$ is clearly a definable set, to which we will refer as the variety associated  with $\Sigma$. The group $G_{\Sigma}$ is sometimes called the coordinate group of the variety. 
    
    It can be easily shown that $V_{\Sigma}$ coincides with the images of $\xx$ by the collection $V_{\F}^{G}(\xx)$ (or later simply $V^{G}$) of all homomorphisms from $G_{\Sigma}$ to $\F$ compatible with the two embeddings of the tuple $a$. In particular, $\F$ has to embed into $G$ for such a solution to exist.
    
    \begin{remark}
    Although $V_{\F}^{G}(\xx)$ clearly depends on the tuple $\xx$, for any fixed $G$ and a distinguished embedding of $\F$ into $G$ there is a definable bijection between the sets $V_{\F}^{G}(\xx)$ and $V_{\F}^{G}(\xx')$ given by different tuples of generators $\xx'$ and $\xx$.
    \end{remark}
    
    Conversely, any finitely presented group equipped with a distinguished embedding of $\F$ into it and a distinguished tuple of generators (rel the image of $\F$) is equal to $G_{\Sigma}$ for any system $\Sigma(x,a)=1$ of equations given by a presentation (the terms in the system are words representing the relators of the presentation).
    
    By \cite{Sela1}, $G_{\Sigma}$ admits finitely many quotients $q_1, \ldots, q_s$ with $q_i: G_{\Sigma} \to L_i$ such that the groups $L_i$ are limit groups and any homomorphism $f: G_{\Sigma} \to \F$ which respects the given embedding of $a$ factors through one of the quotients $q_i$. Equivalently, if the limit group $L_i$ is given the presentation $\langle x^i \mid \Sigma_i(x^i,a) \rangle$, where $x^i=q_i(\xx)$, then    
    the variety $V_{\Sigma}$ over $\F$ is the union of the varieties $V_{\Sigma_{i}}$. It can be shown that varieties whose coordinate groups are limit groups cannot be decomposed into a finite union of proper subvarieties, so that they are the "irreducible components" of varieties over the free group.
    
    Thus without loss of generality when proving lower bounds on the Shelah rank one may assume that the coordinate group is a limit group. 
    
    We now want to give a sufficient condition for such a variety to have infinite Shelah rank. We first recall the following definition (see page 7 in \cite{Sela4}).
    
    \begin{definition} \label{MinimalRankDef} Let $L$ be a limit group equipped with a fixed embedding of the free group $\F(a)$. We say that $L$ has \textbf{minimal rank} relative to $\F(a)$ if there is no epimorphism $\pi: L \to F(a)* \F$  restricting to the identity on $\F(a)$, where $\F$ is a non-trivial free group.
    \end{definition}
    
    \begin{remark} \label{EquivalentDefMinimalRank} Suppose there exists a morphism $\pi: L \to \F(a)$ which is the identity on $a$ and factors through a non-trivial free product $A * B$ of groups $A,B$ where the image of $a$ is contained in $A$ and $B$ admits a non trivial epimorphism to a free group (for example, $B$ is a limit group). Then $L$ does not have minimal rank.
    
    Conversely, if no such morphism exist, $L$ has minimal rank. This gives us an equivalent characterization of minimal rank that will prove itself useful.
    \end{remark}
    
    \begin{proposition} \label{SuperstableDoesNotExpand} Let $G = \langle x, a \mid \Sigma(x,a) \rangle$ be a limit group as above and let $V^{G}$ be the associated variety.
    
    If $G$ does not have minimal rank, then $R^{\infty}(V^G) = \infty$.
    \end{proposition}
    
    \begin{proof} We apply the main result of \cite{PerinSklinosSuperstable}, which shows that given a formula $\phi(x,a)$ over the free group $\F(a_1, \ldots, a_n)$, $\phi$ is superstable only if $$\phi(\F(a_1, \ldots, a_n)) = \phi(\F(a_1, \ldots, a_n, a_{n+1}))$$
    
    The surjective morphism $\pi: G \to \F(a)* \F$ in Definition \ref{MinimalRankDef} corresponds to a solution to the equation $\Sigma(x,a)=1$ which does not lie in $\F(a)$. This shows that the set defined by $\Sigma(x,a)=1$ is strictly bigger in $\F(a)*\F$ than in $\F(a)$, which proves the result.
    \end{proof}

\section{Tower varieties}
  
  \label{Towers}
  The aim of this section is to describe the definable sets for whose Shelah rank we will give a lower bound. They are varieties whose corresponding coordinate groups are hyperbolic towers.
  
  \subsection{Graphs of groups}
    \newcommand{\edin}[1]{\overline{#1}}
    \newcommand{\invf}{\edin{\cdot}}
    
    We assume familiarity with graph of groups decompositions, and will simply recall definitions and fix notations. A detailed introduction can be found in \cite{dicks1989groups} and \cite{serre1974trees}.
    
    A graph $X=(V,E)$ is given by a set $V$ of vertices, a set $E$ of oriented edges, a map $\alpha:E\to V$ assigning an origin to each oriented edge, and an involution $\overline{\cdot}:E\to E$ taking every edge to its inverse, from which it is distinct. 
    
    \begin{definition}
    A \textbf{graph of groups} consists of a connected graph $(E,V)$, the underlying graph, together with the following data:
    \begin{itemize}
    	\item for each $v\in V$, a group $\Gamma_{v}$  called the {\bf vertex group} associated to $v$.
      	\item for each $e\in E$, a subgroup $\Gamma_{e}\leq \Gamma_{\alpha(e)}$  called the {\bf edge group} associated to $e$.
    	\item  for each $e\in E$, an isomorphism $i_{e}:\Gamma_{e}\cong \Gamma_{\edin{e}}$, so that $i_{\edin{e}}=(i_{e})^{-1}$.
    \end{itemize}
    \end{definition}
    As is customary, we will abuse notation and denote such a graph of groups simply by its underlying graph $\Gamma$.
    
    \begin{definition}
    Let $\Gamma$ be a graph of groups and $Z$ a maximal subtree of $\Gamma$.
    Given any choice of presentations $\Gamma_{v}=\langle X_{v}\mid R_{v} \rangle$ of the vertex groups, \textbf{the fundamental group of} $\Gamma$ with respect to $Z$,
    denoted by $\pi(\Gamma,Z)$, is the group given by taking generator set 
    \begin{align*}
     \bigcup_{v\in V} X_{v}\cup \{ t_{e} \}_{e\in E}
    \end{align*}
    together with the relations
        \begin{align*}
        	\bigcup_{v\in V} R_{v}\cup \{t_{e}=1\}_{e\in Z}\cup\{t_e i_{\edin{e}}(g) t_{e}^{-1} g^ {-1}=1 ,\;t_{e}t_{\edin{e}}=1|e\in E, g\in \Gamma_{e}\}
        \end{align*}
    \end{definition}
    
     It can be shown that the natural homomorphisms from $G_{v}$ to the resulting $\pi(\Gamma,Z)$ are injective, thus we think of $G_v$ as a subgroup of the fundamental group and that different choices of presentations for $\Gamma_{v}$ result in equivalent presentations, so that the object $\pi(\Gamma,Z)$ does not depend on said presentations. The isomorphism class of $\pi(\Gamma,Z)$ is in fact independent of the choice of $Z$.   
     
    \begin{definition}
    A \textbf{graph of groups decomposition} of a group $G$ is an isomorphism between $G$ and the fundamental
    group of a graph of groups.
    \end{definition}
    
    An automorphism of a vertex group which respects the conjugacy class of incoming edge groups can be extended to an automorphism of the full fundamental group.
    \begin{lemma}
    \label{VertexAutomorphisms} Let $\Gamma$ be a graph of groups, and $Z$ a maximal subtree of $\Gamma$. Let
    $v$ be a vertex of $\Gamma$ and $\sigma$ an automorphism of $\Gamma_{v}$ such that
    for all $e$ with $\alpha(e)=v$ there exists $c_e \in \Gamma_v$ such that the restriction of $\sigma$ to $\Gamma_e$ is just the conjugation map 	$\iota_{c_e}$. Then $\sigma$ can be extended to an automorphism of $\pi(\Gamma,Z)$.
    \end{lemma}
    \begin{proof}
    	For $w\in V \setminus \{v\}$ let $c_{w}$ be $c_{e}$ where $e$ is the first edge of the unique simple path in $Z$ from $v$  to $w$. For $f\in E \setminus Z$ let $d_{f}$ be $c_{e}$ where $e$ is the first edge of a path $P$ in $Z$ starting at $v$ and ending at $\alpha(f)$. It is easy to check that the map sending $g\in G_{w}$ to $c_w g{c_{w}}^{-1}$ and $t_{f}$ to $d_{f}t_{f}d_{\edin{f}}^{-1}$ for $f\in E(\Gamma) \setminus E(Z)$ extends to an endomorphism of $\pi(\Gamma,Z)$ and that the endomorphism obtained from the same construction by replacing $\sigma$ by $\sigma^{-1}$ and $c_{e}$ by $c_{e}^{-1}$ is an inverse for it.
    \end{proof}

  \subsection{Floors and towers}
    
    \begin{definition} Given a graph of groups $\Gamma$, a \textbf{quadratically hanging group} of $\Gamma$ is a vertex group $\Gamma_{v}$ which is isomorphic to the fundamental group of a compact surface with boundary $\Sigma$, so that:
    \begin{itemize}
    	\item  $\chi(\Sigma) \leq -2$ or $\Sigma$ is a punctured torus;
    	\item there is a bijection $C$ between the sets of edges $e$ incident to $v$ and the set of boundary components of $\Sigma$
    	so that the image of $\Gamma_{e}$ in $\Gamma_{v}$ coincides with that of one of the mutually conjugate embeddings of $\pi_1(C(e))$ in $\pi_1(\Sigma)$.
    \end{itemize}
    \end{definition}
    
    \begin{definition}
    A \textbf{graph of groups with surfaces} $\Lambda$ is a graph of groups $\Lambda$ together with a distinguished set of vertices $V_S$ such that for any $v \in V_S$, the vertex group of $v$ is quadratically hanging. A vertex $v$ in $V_S$ is called a \textbf{surface type vertex} of $\Lambda$. We extend this denomination to vertices in the Bass-Serre tree associated to $\Gamma$ which are mapped to a vertex of $V_S$ by the quotient map.
    \end{definition}
   
    \begin{definition} An \textbf{abelian floor} is given by a tuple $(G,G',\Gamma,r)$ where $r$ is a retraction of a group $G$ onto a subgroup $G'$ and $\Gamma$ a graph of groups decomposition of $G$ with two vertices joined by a single edge, corresponding to a decomposition of $G$ of the form $G=G'*_{\mathbb{Z}}\mathbb{Z}^{2}$ where $\mathbb{Z}$ embeds onto a direct factor of $\mathbb{Z}^2$.
    \end{definition}

    \begin{definition}
    A \textbf{hyperbolic floor} relative to a group $H$ is given by a tuple $(G,G',\Gamma,r)$ where $H \leq G'\leq G$ are groups, and $\Gamma$ is a decomposition of $G$ as a graph of groups with surfaces such that:
    \begin{itemize}
    	\item $\Gamma$ is a bipartite graph between surface type vertices and non surface type vertices;
    	\item There exists a lift $T^0$ of a maximal subtree of $\Gamma$ such that if $v_1, \ldots, v_m$ are the non surface type vertices of $T^0$, then $G'= G_{v_1}* \ldots * G_{v_m}$ and $H \leq G_{v_1}$.
    	\item Either $r$ is a retraction of $G$ onto $G'$ sending every surface type vertex group to a non-abelian image or
    	$G'$ is cyclic and $r$ is a retraction from $G*\mathbb{Z}\to G'*\mathbb{Z}$ with the same property.
    \end{itemize}
    If $\Gamma$ has only one non surface type vertex, we say the floor is \textbf{connected}.
    \end{definition}

    \begin{definition} 	A \textbf{tower structure} $\mathcal{T}$ of a group $G$ over a subgroup $H$,
    consists of a finite sequence of groups $G=G_{0}\geq G_{1} \geq \cdots \geq G_{n} \geq H$ such that for each $i$, there exists $\Gamma_i$ and $r_i$ such that $(G_i, G_{i+1}, \Gamma_i, r_i)$ is either an abelian or a hyperbolic floor relative to $H$, and $G_n= H * \F_r * S_1 * \ldots * S_q$ where $\F_r$ is a free group of rank $r \geq 0$, and each $S_i$ is the fundamental group of a closed surface with Euler characteristic at most $-2$.
    A group is called a \textbf{tower} if it admits (at least one) tower structure.
    A tower structure is said to be \textbf{hyperbolic} if none of the floors are of abelian type.
    \end{definition}

    \begin{remark} \label{refinement}Let $(G, G',\Gamma, r)$ a hyperbolic floor relative to a subgroup $H$ of $G$ and 
    	$S_{1},S_{2},\dots, S_{s}$ the vertex groups of $\Gamma$. Then for every permutation $\sigma\in S_{s}$ it is possible to 
    	define a finite sequence $G=G_0 \geq G_1 \geq \ldots \geq G_s = G'$ and for each $i$ a hyperbolic floor $(G_{i-1}, G_{i}, \Gamma_i, r_i)$ relative to $H$ such that $\Gamma_i$ has a unique surface type vertex with stabilizer conjugate to $S_{\sigma(i)}$ and with the same peripheral structure.
    \end{remark}
    
    \begin{remark} \label{TowersLimitHYpTowersHyp}By Bestvina and Feighn's combination theorem, one can show that a hyperbolic tower over a hyperbolic group is itself hyperbolic (see \cite{BFCombinationTheorem}).
    
    It is a result of Sela that towers over a non trivial free group $\F(a)$ are limit groups relative to $\F(a)$ (see \cite{Sela6}). In particular they admit epimorphisms to $\F(a)$ which restrict to the identity on $a$.
    \end{remark}
    
    Towers appear in the following description provided by Sela in \cite{Sela6} of the finitely generated models of the theory of
    the free group.
    
    \begin{theorem} 
    A finitely generated group is a model of the theory of non abelian free groups if and only if
    it is non abelian and admits a hyperbolic tower structure over the trivial subgroup.
    \end{theorem}
    
    The following result also follows from the work of Sela:
    
    \begin{theorem}
    If a non abelian torsion free hyperbolic group $G$ admits the structure of a tower over a non abelian subgroup $H$, then the inclusion $H\hookrightarrow G$ is an elementary embedding.
    \end{theorem}
    In fact, the converse also holds by \cite{PerinElementary}.
    
    \subsection{Envelopes}
        
    As mentioned in the introduction, towers play yet another significant role in the works of Sela. Indeed, if $T$ is a tower and $V_T$ the associated variety, there is a natural notion of genericity of a sequence of points $(x_n)_{n \in \N}$ in $V_T$ (given by test sequences). 
    
    Sela proves in \cite{SelaImaginaries} that given any definable set $X \subseteq M^k$ in the theory of free groups, it is possible to find a finite set of towers whose generic sequences give sequences of elements which eventually must lie in $X$ - Sela calls such a set of towers an \textbf{envelope} of $X$. 
    
    More precisely, consider inclusions of towers of the form $T_1 \leq T_2 \leq T_3$, and adapted presentations of the form $\langle y_{T_i},x,a \mid \Sigma_{T_i}(y_{T_i},x,a)\rangle$ (where $y_{T_{i}}$ is a subtuple of $y_{T_{i+1}}$). Denote by $p(V_T)$ the projection of $V_T =\{(y_T,x) \mid \Sigma(y_T,x) \}$ on the variables $x$ given by the formula $\exists y_T \; \Sigma_T(y_T,x,a)$. Note that if $T_1 \leq T_2 \leq T_3$ we have $p(V_{T_3}) \subseteq p(V_{T_2}) \subseteq p(V_{T_1})$. 
    
    An envelope of a set $X$ defined by a formula on the variables $x$ is a finite height-2 tree of inclusions of towers with branches of the form $T \leq T' \leq T''$ as above such that
    \begin{enumerate}
    	\item for any $x \in X$, there exists a node $T$ of height $0$ in the tree such that $x$ lies in $p(V_{T})$ but not in $p(V_{T'})$ for any child $T'$ of $T$, or $x$ lies in fact in $p(V_{T''})$ for some grandchild $T''$ of $T$;
    	\item if for some node $T$ of height $0$ there exists a generic sequence $(y_{T}^n, x^n)$ in $V_{T}$ which does not extend to $V_{T'}$ for any of the children $T'$ of $T$, or if there exists a generic sequence $(y_{T''}^n, x^n)$ for $V_{T''}$ where $T''$ is a grandchild of $T$, then for all $n$ large enough $x^n$ lies in $X$.
    \end{enumerate}

There are reasons to hope that this notion of genericity is compatible with Shelah rank, in the sense that the Shelah rank of a superstable definable set can be computed in terms of the Shelah rank of the towers in one of its envelopes. For example, in the special case where the envelope of $X$ consists of a single tower $T$, the conjecture is that $\SR(X) = \SR(V_T)$. 
   
   For this reason computing Shelah rank of tower varieties seems like an essential first step in the understanding of Shelah rank for general definable sets. 
     
  \subsection{Minimal rank tower varieties}

    \begin{definition} Let $T$ be a (hyperbolic) tower over the free group $\F=\F(a)$.
    A \textbf{(hyperbolic) tower variety} associated to $T$ is a definable set $V^T$ of the form $\{ x \in \F \mid \Sigma_T(x, a)\}$ where $\langle x, a \mid \Sigma_T(x,a) \rangle$ is a presentation for $T$.
    \end{definition}

    \begin{lemma} \label{MinRankTowerProperties}  Let $T$ be a tower over $\F = \F(a)$. If $T$ has minimal rank, then the following hold for any floor $(T_i, T_{i+1}, \Gamma_i, r_i)$ of any tower structure for $T$ over $\F$:
    	\begin{enumerate}
    		\item the floor is connected;
    		\item $T_{i+1}$ is freely indecomposable relative to $\F(a)$;
    		\item if the floor is hyperbolic,  then $\Ker(r_i)$ does not contain any element corresponding to a simple closed curve on a surface of $\Gamma_i$.
    	\end{enumerate} 
    \end{lemma}
    
    \begin{proof} The retraction $T \to \F(a)$ factors through the retraction of $T$ onto $T_{i+1}$.  If one of the first two hypotheses does not hold, $T_{i+1}$ is a free product $A*B$ of limit groups relative to $A$ in which $a$ is contained in the first factor. 
    	
   If the third hypothesis does not hold, the retraction $T \to \F(a)$ factors through the composition of the retraction of $T$ on $T_i$ with the quotient map of $T_i$ by the normal subgroup generated by the element of $\Ker(r_i)$ corresponding to a simple closed curve. Again, this is a free product as above.
    	
   By Remark \ref{EquivalentDefMinimalRank}, this implies that the tower does not have minimal rank, a contradiction.
    \end{proof}

    As a consequence, minimal rank towers have the following nice property
    \begin{lemma} \label{MinimalRankOnlyOneFloorStructure} Let $T$ be a hyperbolic tower over $\F(a)$ of minimal rank, with top floor $(T, T_0, \Gamma_0, r_0)$ such that $\Gamma_0$ has only one surface. Then $T$ does not admit a free product decomposition for which $T_0$ is contained in one of the factors, nor does it admit a floor structure of the form $(T,T_1, \Gamma_1, r_1)$ for any subgroup $T_1$ which properly contains $T_0$.
    \end{lemma}
    This will help us in Section \ref{DefOfOrbitsSec} to find plenty of definable subsets, which will enable us to give a lower bound on the rank.
    
    \begin{proof} Suppose first that there is a non trivial decomposition $T = A*B$ with $T_0 \leq A$. Now $T$ is a limit group over $\F(a)$ so this contradicts minimal rank.
    
    Suppose now that $T$ admit a floor structure of the form $(T,T_1, \Gamma_1, r_1)$ for some subgroup $T_1$ which properly contains $T_0$. Denote by $\Sigma_0$ the (single) surface of the floor decomposition $\Gamma_0$.
    
    First, note that $\Gamma_0$ and $\Gamma_1$ have exactly one non surface type vertex by Lemma \ref{MinRankTowerProperties}, which thus must have vertex groups $T_0, T_1$ respectively. We see that $\Gamma_0$ is also the JSJ decomposition of $T$ relative to $T_0$ (for example by applying Lemma 5.3 in \cite{GuirardelLevittJSJI}). Thus $\Gamma_1$ can be obtained from $\Gamma_0$ by refining at the surface type vertex into a graph $\hat{\Gamma}_0$ and then collapsing all the edges except those created in the refinement. 
    
 Now $T_1$ is the unique non surface type vertex group of $\Gamma_1$, thus it is the fundamental group of a subgraph of groups $\hat{\Gamma}'_0$ of $\hat{\Gamma}_0$ which contains at least one vertex group $H$ coming from a piece $\Pi_1$ of $\Sigma_0$. This group $H$ is the fundamental group of the surface group with boundary $\Pi_1$, but the conjugacy class of the boundary subgroup not adjacent to the vertex corresponding to $T_0$ does not correspond to an edge group in $\hat{\Gamma}'_0$. This shows that $T_1$ can be written as a free product of $T_0$ together with a non trivial free group - this contradicts the minimal rank assumption since the retraction $r_1$ gives a morphism $T \to T_1$ which fixes $\F(a)$.
    \end{proof}

  \subsection{Statement of the main result}
    
    We can now state the main results of this paper.
    \begin{theorem} \label{MainResultOneFloor}
    Let $T$ be a hyperbolic tower over $\F(a)$ with one floor, with a single surface $\Sigma$ in the associated floor decomposition. Denote by $V^T$ the associated variety.
    Then $R_{\F(a)}^{\infty}(V^T) \geq \omega^{-\chi(\Sigma)}$.
    \end{theorem}
    
    More generally we get
    \begin{theorem} 
    \label{MainResultSeveralFloors} Let $T$ be a hyperbolic tower over $\F(a)$ with tower structure $T=T_m \geq T_{m-1} \geq \ldots \geq T_1 \geq T_0=\F(a)$, and suppose that for each floor $(T_i, T_{i+1}, \Gamma_i, r_i)$, the floor decomposition $\Gamma_i$ has a unique surface $\Sigma_i$. Denote by $V^T$ the associated variety.
    
    Then $R_{\F(a)}^{\infty}(V^T) \geq \omega^{-\chi(\Sigma_m)} + \ldots + \omega^{-\chi(\Sigma_1)}$.
    \end{theorem}    
    
    Note that a given tower group $T$ may admit several tower structures (with one surface per floor), and these may give different lower bounds. 
    Using the fact that we have $\alpha+\beta=\beta\neq\beta+\alpha$ for any two ordinals $\alpha<\beta$, Remark \ref{refinement} can be used to provide many examples in which this is the case. 
    
\section{Definability of orbits} \label{DefOfOrbitsSec}
  
  Let $T$ be a hyperbolic tower over $\F(a)$. Witnessing the lower bound given in Theorem \ref{MainResultOneFloor} and \ref{MainResultSeveralFloors} requires finding a host of definable subsets of the variety $V^T$. We want to apply the condition on automorphisms of small models given by Lemma \ref{automorphisms and rank} to compute lower bounds on the Shelah rank. 
  
  We can in fact take as small model the hyperbolic tower $T$ itself, since non abelian hyperbolic towers over the trivial group are models of the theory of free groups. 
  
   Orbits of a generating tuple of $T$ by subgroups of $Aut_{\F}(T)$ are a natural candidate, since they clearly lie in $V^T$, and behave well with respect to the action of the group $Aut_{\F}(G)$. On the other hand, it is not a priori obvious that they are definable. Here is where the minimal rank condition comes into play.

  \begin{definition} \label{HSigmaDef}
  	  Let $\Sigma$ be a surface with boundary, and pick $B$ a boundary subgroup of $\pi_1(\Sigma)$. We denote by $H_{\Sigma, B}$ the group of automorphisms $\pi_1(\Sigma)$ which restrict to the identity on $B$ and to a conjugation on any other boundary subgroups. 
   \end{definition} 	  
We sometimes simply denote $H_{\Sigma, B}$ by $H_{\Sigma}$ when the choice of $B$ is clear from the context.

\begin{remark} If $\Lambda$ is a graph of group with a surface type vertex $v$ corresponding to $\Sigma$, any automorphism in $H_{\Sigma}$ extends to an automorphism of $\pi_1(\Lambda)$ which restricts to the identity on the vertex group adjacent to $v$ by the edge corresponding to $B$.
\end{remark}
  	  
\begin{ex} \label{EmbeddingsHSigma}
	\begin{enumerate}
			\item  Suppose $\Sigma, B$ are as above, and let $\Sigma_{0}$ be a $\pi_1$- embedded subsurface of $\Sigma$. If $\Sigma_0$ contains the boundary component corresponding to $B$, let $B=B_0$. Otherwise, let $\Delta_{\Sigma_0}$ be the graph of groups decomposition of $\pi_1(\Sigma)$ dual to the boundary components of $\Sigma_0$. Pick a pair of adjacent vertex groups $S_0\simeq \pi_1(\Sigma_0)$ and $U$ corresponding respectively to $\Sigma_0$ and to a connected component of its complement in such a way that $B \leq U$, and define $B_0 = S_0 \cap U$. Then any element of $H_{\Sigma_0, B_0}$ extends to an element of $H_{\Sigma, B}$, this gives an embedding $H_{\Sigma_0, B_0} \leq H_{\Sigma,B}$.
			
			\item  If $(T, T_0, \Gamma, r)$ is a \textbf{connected} hyperbolic floor such that $\Gamma$ has a unique surface $\Sigma$, and if $\Sigma_{0}$ is a $\pi_1$- embedded subsurface of $\Sigma$, we consider the graph of groups obtained by refining the vertex corresponding to $\Sigma$ by $\Delta_{\Sigma_0}$ and then collapsing all the edges which are not adjacent to the vertex corresponding to $\Sigma_0$. The two vertices of this new graph $\Gamma_{\Sigma_0}$ each define a conjugacy class of subgroups of $T$, we pick representatives corresponding to subgroups $A_{\Sigma_0}$ and $\pi_1(\Sigma_0)$ whose intersection is a boundary subgroup $B_0$ of $\pi_1(\Sigma_0)$. Any element of $H_{\Sigma_0, B_0}$ extends to an automorphism of $T$ which restricts to the identity on $A_{\Sigma_0}$ thus we get an embedding $H_{\Sigma_0} \leq \Aut_{A_{\Sigma_0}}(T)$. 
	\end{enumerate}
\end{ex}
			
			\begin{lemma} Suppose we are in the setting of Example \ref{EmbeddingsHSigma}(2) above, and that $B_0$ is maximal cyclic in $A_{\Sigma_0}$ (this is the case for example if $B_0$ is not a boundary subgroup of $\pi_1(\Sigma)$).
				Then the embedding $H_{\Sigma_0, B_0} \to \Aut_{A_{\Sigma_0}}(T)$ is in fact an isomorphism.
			\end{lemma}
			
			\begin{proof}
			Indeed, if $h \in \Aut_{A_{\Sigma_0}}(T)$, it fixes elements of $A_{\Sigma_0}$, thus its restriction to the subgroup $\pi_1(\Sigma_0)$ is the identity on $B_0$ and a conjugation on other boudary subgroups of $\pi_1(\Sigma_0)$, hence all boundary subgroups are elliptic in $\Gamma_{\Sigma_0}$. The subgroup $h(\pi_1(\Sigma_0))$ inherits a decomposition from $\Gamma_{\Sigma_0}$ which in turn induces a decomposition $\Delta_{\Sigma_0}$ on $\pi_1(\Sigma_0)$ which is dual to a set of simple closed curves on $\Sigma_0$. The vertex group $S$ of $\Delta_{\Sigma_0}$ containing $B_0$  must be sent by $h$ to a vertex group of $\Gamma_{\Sigma_0}$ containing $B_0$: by injectivity of $h$ it cannot be sent to $A_{\Sigma_0}$, and since $B_0$ is maximal cyclic the only other possibility is $\pi_1(\Sigma_0)$ itself. Now $S$ corresponds to a subsurface of $\Sigma_0$ and its maximal boundary subgroups are sent to maximal boundary subgroups of $\pi_1(\Sigma_0)$:  by Lemma 3.13 in \cite{PerinElementary}, $h$ is in fact an isomorphism of surface groups between $S$ and $\pi_1(\Sigma_0)$, which implies that the subsurface corresponding to $S$ is in fact $\Sigma_0$ itself. This implies the result.
		\end{proof}
 
  We will adopt the convention that by a subsurface, we always mean a $\pi_1$-embedded subsurface.
  
  \begin{proposition} \label{DefinabilityOfOrbits}
  Let $T$ be a hyperbolic tower with minimal rank over the parameter group $\F(a)$. Suppose the top floor structure has a unique surface $\Sigma$. Let $\Sigma_0$ be a proper subsurface of $\Sigma$.
  Then the orbit of any tuple of elements of $T$ by $H_{\Sigma_0}$ (identified to $\Aut_{A_{\Sigma_0}}(T)$ as by the previous lemma) is definable over $A_{\Sigma_0}$.
  \end{proposition}
  
  This follows from Theorem 5.3 of \cite{PerinSklinosForking}, which states that if $G$ is a torsion-free hyperbolic group which is concrete over a subset $A$, then the orbit of any tuple of elements of $G$ under $\Aut_A(G)$ (the group of automorphisms of $G$ which fix $A$ pointwise) is definable over $A$. 
  A group $G$ is said to be concrete over $A$ if $A$ is not contained in a proper free factor of $G$, and $G$ does not admit the structure of an extended hyperbolic floor over $A$.  
  
  In the context above, Lemma \ref{MinimalRankOnlyOneFloorStructure} implies that $T$ is concrete with respect to $A_{\Sigma_0}$ and Proposition \ref{DefinabilityOfOrbits} thus follows immediately.

\section{Free products and the Shelah rank}
  
  \label{FP&SR} 
  For the rest of this section we will fix a first-order structure $M$ and a set $A$ of parameters.
  
  Given a subset $\mathcal{H}\subset Aut_A(M)$ and a definable set $X \subset M^k$, we denote by
  $\mathcal{H}(X)$ the set $\{h(u) \mid h\in\mathcal{H}, u \in X \}\subset M^{k}$.
  Given subsets $\mathcal{H}_1,\mathcal{H}_2\subset Aut_A(M)$, we denote by $\mathcal{H}_1\circ\mathcal{H}_2$
  the set $\{h_1\circ h_2\,|\,h_1\in\mathcal{H}_1,\,h_2\in\mathcal{H}_2\}$. In such an expression we might write $h$ in place of the singleton $\{h\}$.
  
  \begin{lemma}
  \label{preservation of definability}
  Let $M$ be a first order structure, and let $v \subseteq M$ be a finite tuple. Let $\mathcal{H}\subset Aut_A(M)$ be such that 
  the set $\mathcal{H} \cdot v = \mathcal{H}(\{v\})$ is definable over some set $B$ of parameters.
  Suppose the set $X\subset M^{k}$ is definable over $v$.
  
  Then $\mathcal{H}(X)$ is a subset of $M^{k}$ which is definable over $B$.
  \end{lemma}
  \begin{proof}
  Let $\phi(x,B)$ be a defining formula for $\mathcal{H}\cdot v$ and $\psi(y,v)$ a defining formula for $X$. We can then write:
  \begin{align*}
  	\bunion{h\in\mathcal{H}}{h(X)}{}=\bigcup_{h\in\mathcal{H}}h(\psi(M,v))=&\bigcup_{h\in\mathcal{H}}\psi(M,h(v))=
  	\\
  	=\{x\in M^{k}\,|\,\exists z\,\phi(z,B)&\wedge\psi(y,z)\,\}
  \end{align*}

  \end{proof}
  \begin{corollary} \label{ComposingSetsOfAut}
  Let $\mathcal{H}_{1}$ and $\mathcal{H}_{2}$ be subsets of $Aut_A(M)$ such that for $i\in\{1,2\}$ the set $\mathcal{H}_{i}\cdot v$ is definable over $v$. Then $(\mathcal{H}_{1}\circ\mathcal{H}_{2})\cdot v$ is definable over
  $v$. 
  \end{corollary}

To witness a lower bound in Shelah rank, we will consider sequences of definable sets whose elements are the images of a given tuple by subsets of some automorphism groups. The purpose of the following definition is to give conditions on these subsets of the group of automorphisms we choose that will be sufficient to enable us to give lower bounds on Shelah rank. 

\begin{definition} \label{WitnessingSequenceDef} Fix a subset $A$ of parameters in a model $M$. Let $v$ be a tuple in $M$ and let $\GG \leq \Aut_A(M)$.

If $\alpha$ is a limit ordinal, a \textbf{witnessing sequence in $\GG$ of length $\alpha$ relative to $v$} is a sequence
$\{\mathcal{S}_{\beta}\}_{\beta < \alpha}$ of non-empty subsets of $\GG$ such that
  \begin{enumerate}[(i)]
  \item $S_{0}=\{Id_{M}\}$. \label{0_c}
  \item The set $S_{\beta}\cdot v$ is definable over $v$ for all $\beta<\alpha$. \label{definability_c}
 \item \label{succesor_c} If $ֿ\beta < \alpha$, then there is a sequence $\{\sigma^{\beta}_{n}\}_{n\in\N}$ of elements from $\GG$ such that:
  	\begin{enumerate}[a)]
  		\item The sets $\{\sigma^{\beta}_{n} \circ \mathcal{S}_{\beta}\}_{n\in\N}$ are pairwise disjoint;
  		\item $\sigma^{\beta}_{n}\circ\mathcal{S}_{\beta}\subset\mathcal{S}_{\beta+1}$ for any $n\in\N$.
  		\item $\sigma^{\beta}_n \circ (\sigma^{\beta}_0)^{-1} \circ \mathcal{S}_{\beta+1} = \mathcal{S}_{\beta+1}$ for all $n \in \N$.
  	\end{enumerate}
  \item If $\beta < \alpha$ is a limit ordinal, then there exists a sequence $\beta_k \to \beta$ with $\beta_k<\beta$ such that $\mathcal{S}_{\beta_k} \subseteq \mathcal{S}_{\beta}$. \label{inclusion_c}
  \end{enumerate}
\end{definition}

The following lemma shows how witnessing sequences help give lower bounds on Shelah ranks:
  \begin{lemma} \label{WitnessingSequence} Let $A$ be a subset of parameters in a model $M$, and let $v$ be a tuple in $M$. Let $X\subset M^{k}$ be definable over $v$ and non empty.
  
  Suppose $\GG \leq \Aut_A(M)$ be such that if $h, g \in \GG$ satisfy $h(X) \cap g(X) \neq \emptyset$, then $h=g$, and suppose $\GG(X)$ is contained in a definable set $Y$.
  
  If there exist a witnessing sequence in $\GG$ of length $\alpha$ relative to $v$, then $\SR_A(Y)\geq \SR_A(X) + \alpha$.
  \end{lemma}

  \begin{proof} Note that by Lemma \ref{preservation of definability} and condition (\ref{definability_c}), the set $\mathcal{S}_{\beta}(X)$ is definable over $v$ for all $\beta < \alpha$. 
  
  Let us prove by induction that $\SR_A(\mathcal{S}_{\beta}(X))\geq \SR_A(X) + \beta$ for all $\beta < \alpha$. It clearly holds for $\beta=0$ by Condition (\ref{0_c}). Now fix $\beta$ and suppose that for every $\gamma < \beta$ we have $\SR_A(\mathcal{S}_{\gamma}(X))\geq \SR_A(X) + \gamma$. 
  
  Suppose first $\beta=\gamma+1$. By condition (\ref{succesor_c}),  there exists a sequence $\sigma^{\gamma}_n$ such that the sets $\sigma^{\gamma}_n(\mathcal{S}_{\gamma}(X))$ are pair-wise disjoint subsets of $\mathcal{S}_{\gamma+1}(X)$. The set $Z = \sigma^{\gamma}_{0}(\mathcal{S}_{\gamma}(X))$ is definable. Since $\SR_A$ is invariant under the action of automorphisms that preserve $A$, we have $\SR_A(\sigma^{\gamma}_0 (\mathcal{S}_{\gamma}(X)))\geq \SR_A(X) + \gamma$. Finally, by (\ref{succesor_c}) we have that the automorphisms $\sigma^{\gamma}_n \circ (\sigma^{\gamma}_0)^{-1}$ preserve $\mathcal{S}_{\gamma+1}(X) = \mathcal{S}_{\beta}(X)$, so by Lemma \ref{automorphisms and rank} we have $\SR_A(\mathcal{S}_{\beta}(X)) \geq \SR_A(\mathcal{S}_{\gamma}(X)) +1$ and hence
  $$\SR_A(\mathcal{S}_{\beta}(X)) \geq \SR_A(X) +\gamma +1 = \SR_A(X) + \beta$$
  
  Suppose now that $\beta$ is a limit ordinal. Condition (\ref{inclusion_c}) gives a sequence $\beta_k \to \beta$ with $\beta_k < \beta$ and $\mathcal{S}_{\beta_k} \subseteq \mathcal{S}_{\beta}$, thus and $\mathcal{S}_{\beta_k}(X) \subseteq \mathcal{S}_{\beta}(X)$. Now by induction hypothesis, $\SR_A(\mathcal{S}_{\beta_k} (X)) \geq \SR_A(X) + \beta_k$  so we get that  $\SR_A(\mathcal{S}_{\beta}(X)) \geq \SR_A(X) + \beta_k$ for all $k$. Therefore $\SR_A(\mathcal{S}_{\beta}(X)) \geq \SR_A(X) + \beta$.
  
Thus we know that $\SR_A(\mathcal{S}_{\beta}(X))\geq \SR_A(X) + \beta$ for all $\beta < \alpha$. But $Y$ is definable and contains all the above sets, thus $\SR_A(Y) \geq \SR_A(X)+\alpha$.
  \end{proof}

  \newcommand{\h}[0]{\tau}
  \newcommand{\hh}[0]{\rho}

  \begin{lemma}
  \label{rank times omega} We fix a subset $A$ of parameters in a model $M$, and $v$ a tuple in $M$. 
  Let  $\GG \leq \Aut_A(M)$.  
  
  Let $\mathcal{K}\leq \GG \leq \Aut_A(M)$ be such that $\mathcal{K} \cdot v$ is definable over $v$. Suppose the following holds
  \begin{enumerate}
  	\item there exists a witnessing sequence in $\mathcal{K}$ of length $\alpha$ relative to $v$. 
  	\item there exists an element $\tau \in\GG\setminus\{1\}$ such that $\tau(v)$ is definable over $v$ and the subgroup of $\GG$ generated by $\mathcal{K}$ and $\tau$ is isomorphic to the free product $\subg{\tau}\frp\mathcal{K}$
  \end{enumerate}
  Then there exists a witnessing sequence in $\GG$ of length $\alpha \cdot \omega$ relative to $v$.
  \end{lemma}
  
  \begin{proof}
  We extend the sequence $(\mathcal{S}_{\beta})_{\beta < \alpha}$ in $\mathcal{K}$ to one in $\GG$ over the index set $\alpha\cdot\omega$ as follows. 
  Any ordinal $\alpha \leq \beta<\alpha\cdot\omega$ is of the form $\beta=\alpha\cdot m+\delta$, where $1 \leq m <\omega$, and $\delta < \alpha$. If $\delta \neq 0$ we set 
  \begin{align*}
  	\mathcal{S}_{\beta} = \mathcal{S}_{\delta}\circ(\h\circ\mathcal{K})^{m}
  \end{align*}
  and if $\delta=0$ we let 
  \begin{align*}
  \mathcal{S}_{\beta} = \mathcal{K} \circ (\h\circ\mathcal{K})^{m-1}.
  \end{align*}
  
  We claim that the resulting sequence $(\mathcal{S}_{\beta})_{\beta < \alpha\cdot\omega}$ is a witnessing sequence relative to $v$.  Conditions (\ref{0_c}) is immediate. Condition (\ref{definability_c}) is easily shown by iteratively applying Corollary \ref{ComposingSetsOfAut}. 
  
  Let us now check Condition (\ref{succesor_c}). Take $\beta < \alpha \cdot \omega$, where $\beta = \alpha \cdot m + \delta$. 
  
  We apply Condition (\ref{succesor_c}) of the witnessing sequence for $\delta < \alpha$: this yields a sequence $(\sigma^{\delta}_n)_{n\in\omega}\subset \mathcal{K}$ such that $a)$ the sets $\{\sigma^{\delta}_n \circ {\mathcal S}_{\delta}\}_{n \in \N}$ are pairwise disjoint; $b)$ $\sigma^{\delta}_{n}\circ\mathcal{S}_{\delta}\subseteq \mathcal{S}_{\delta +1} $ for any $n\in\N$ and $c)$ $\sigma^{\delta}_n \circ (\sigma^{\delta}_0)^{-1} \circ \mathcal{S}_{\delta+1} = \mathcal{S}_{\delta+1}$ for all $n \in \N$.
  
  Suppose first $\delta \neq 0$: then note that the sets $\sigma^{\delta}_n \circ \mathcal{S}_{\beta} = \sigma^{\delta}_n \circ \mathcal{S}_{\delta} \circ (\tau \circ \mathcal{K})^m$ are pairwise disjoint by condition $a)$ and uniqueness of the normal form in a free product, they are subsets of $\mathcal{S}_{\delta +1} \circ (\tau \circ \mathcal{K})^m = \mathcal{S}_{\beta+1}$ by $b)$ above, and $\sigma^{\delta}_n \circ (\sigma^{\delta}_0)^{-1} \circ \mathcal{S}_{\beta +1} = \sigma^{\delta}_n \circ (\sigma^{\delta}_0)^{-1} \circ \mathcal{S}_{\delta +1} \circ (\tau \circ \mathcal{K})^m = \mathcal{S}_{\beta +1} $. Thus if we set $\sigma^{\beta}_n = \sigma^{\delta}_n$ we see that Condition (\ref{succesor_c}) is satisfied. 
  
  If on the other hand $\delta =0$, we set $\sigma^{\beta}_n = \sigma^0_n \circ \h$: then the sets $\sigma^{\beta}_n \circ \mathcal{S}_{\beta} = \sigma^{0}_n \circ \tau \circ \mathcal{K} \circ (\tau \circ \mathcal{K})^{m-1}$ are pairwise disjoint by the uniqueness of the normal form in a free product, they are subsets of $\mathcal{S}_1 \circ (\tau \circ \mathcal{K})^m = \mathcal{S}_{\beta+1}$ by $b)$ above, and $\sigma^{\beta}_n \circ (\sigma^{\beta}_0)^{-1} \circ \mathcal{S}_{\beta +1} = \sigma^{0}_n \circ \h \circ \h^{-1} \circ (\sigma^{0}_0)^{-1} \circ \mathcal{S}_{1} \circ (\tau \circ \mathcal{K})^m = \mathcal{S}_{\beta +1} $ since $\sigma^{0}_n \circ (\sigma^{0}_0)^{-1} \circ \mathcal{S}_{1} = \mathcal{S}_{1}$ by condition (\ref{succesor_c}) for $\delta=0$. Thus with our choice of $\sigma^{\beta}_n$ we see that Condition (\ref{succesor_c}) is satisfied for $\beta$. 
 
 Finally we check condition (\ref{inclusion_c}). Suppose thus that $\beta$ is a limit ordinal with $\beta < \alpha \cdot \omega$. There are two cases: either $\beta = \alpha \cdot m + \delta$ for some $m < \omega$ and $\delta$ a limit ordinal with $\delta < \alpha$, or $\beta = \alpha \cdot m$. 
 
 If $\beta = \alpha \cdot m + \delta$, then Condition (\ref{inclusion_c}) applied to $\delta$ yields a sequence $\delta_k \to \delta$ with $\delta_k < \delta$ such that  $\mathcal{S}_{\delta_k} \subseteq \mathcal{S}_{\delta}$. Now the sequence $\beta_k = \alpha \cdot m + \delta_k$ tends to $\beta$ from below as $k$ goes to infinity, and 
 $$\mathcal{S}_{\beta_k} = \mathcal{S}_{\delta_k} \circ (\tau \circ \mathcal{K})^m  \subseteq  \mathcal{S}_{\delta} \circ (\tau \circ \mathcal{K})^m = \mathcal{S}_{\beta}.$$
 
 If $\beta = \alpha \cdot m$, we can take any sequence $\alpha_k  \to \alpha$ with $\alpha_k < \alpha$, and we know that $\mathcal{S}_{\alpha_k} \subseteq \mathcal{K} = \mathcal{S}_{\alpha}$. Then let $\beta_k = \alpha \cdot (m-1) + \alpha_k$, we have that $\beta_k$ tends to $\beta$ from below as $k$ goes to infinity, and 
 $$ \mathcal{S}_{\beta_k} = \mathcal{S}_{\alpha_k} \circ (\h \circ \mathcal{K})^{m-1} \subseteq \mathcal{K} \circ (\h \circ \mathcal{K})^{m-1} = \mathcal{S}_{\beta}$$

The sequence we produced is thus indeed a witnessing sequence for $v$.
  \end{proof}

  We now want an iterated version of this lemma
  \begin{corollary}
  \label{rank times omega iterated} We fix a subset $A$ of parameters in a model $M$, and $v$ a tuple in $M$. Let  $\GG \leq \Aut_{A}(M)$. 
 
  Let $\mathcal{K}_0 =\{\id_M\} \leq \ldots \leq \mathcal{K}_m =\GG\leq Aut_A(M)$ be a sequence of subgroups such that the sets $\mathcal{K}_j \cdot v$ are definable over $v$ for $0\leq j<m$.
  
  Suppose that for any $j$, there exists some element $\h\in\mathcal{K}_{j+1}\setminus\{1\}$ such that $\tau(v)$ is definable over $v$ and the subgroup of $\GG$ generated by $\mathcal{K}_j$ and $\h$ is isomorphic to the free product $\subg{\h}\frp\mathcal{K}_j$
  
  Then there exists a witnessing sequence in $\GG$ of length $\omega^m$ relative to $v$.
  \end{corollary}
  
  \begin{proof} We prove by induction on $l$ that there is a witnessing sequence $(\mathcal{S}_{\beta})_{\beta < \omega^l}$ of length $l$. For $l=0$ this is immediate. Suppose it is true for $l<m-1$, let us prove it for $l+1$: we apply Lemma \ref{rank times omega} with $\mathcal{K} = \mathcal{K}_l$, $\mathcal{G} = \mathcal{K}_{l+1}$. 
  \end{proof}

\section{Proof of the main results}\label{LoBo}
  
  To prove the main results, which give a lower bound on the Shelah rank of hyperbolic tower varieties, we will use the following proposition, which will be proved in Section \ref{FreeProductSec}. 
  
    \begin{proposition}
  	\label{FreeProduct} Let $\Sigma_1$ be an orientable compact surface with $b\geq 1$ boundary components and genus $g$, where $3g+b\geq 4$. 
  	
  	Fix a boundary subgroup $B$ of $\pi_1(\Sigma_1)$ and let $H_{\Sigma_1}$ be the subgroup of $Aut(\pi_1(\Sigma_{1}))$ consisting of automorphisms which restrict to the identity on $B$ and to a conjugation on the other boundary components. 
  	
  	Let $\Sigma_0$ be a non-empty proper subsurface of $\Sigma_{1}$ such that $\Sigma_1 \setminus \Sigma_0$ is connected and not parallel to a boundary component. Let $H_{\Sigma_0}$ be the subgroup of  $H_{\Sigma_1}$ defined as in Example \ref{EmbeddingsHSigma}.
   	
  	Then there exists an element $\tau \in H_{\Sigma_1}$ such that $\langle H_{\Sigma_0}, \tau \rangle= H_{\Sigma_0} * \langle \tau \rangle$.
  \end{proposition}
  
  \subsection{Proof of Theorem \ref{MainResultOneFloor}}

    \begin{lemma} \label{SequenceFreeProducts} Let $\Sigma$ be a closed surface with boundary, let $B$ be a boundary subgroup of $\pi_1(\Sigma)$ and suppose $m= -\chi(\Sigma)$ satisfies $m \geq 2$, or $m=1$ and $\Sigma$ is a once punctured torus. Let $H_{\Sigma} = H_{\Sigma, B}$ be the group of automorphisms of $\pi_1(\Sigma)$ which fix $B$ and preserve the conjugacy class of the other boundary subgroups as in Definition \ref{HSigmaDef}. Let $v$ be a generating tuple for $\pi_1(\Sigma)$.
    
    Then there exists a sequence $H_0=\{1\} \leq H_1 \leq \ldots \leq H_{m}= H_{\Sigma}$ such that
    for each $i$, there is an element $\tau_i \in H_i$ for which the subgroup generated by $H_{i-1}$ and $\tau_i$ is isomorphic to the free product $H_{i-1} \frp \langle \tau_i \rangle$.
    \end{lemma}
    
    \begin{proof}
    We pick a sequence of subsurfaces $\Sigma_0 \subset \Sigma_1 \subset \ldots \subset \Sigma_m = \Sigma$ such that
    \begin{enumerate}
    	\item $\Sigma_0$ is an annulus
    	\item $\chi(\Sigma_{i}) = -i$;
    	\item for each $i$, the complement of $\Sigma_i$ in $\Sigma_{i+1}$ is connected and not parallel to a boundary component.
    \end{enumerate}
 
    For each $i<m$, we pick a boundary subgroup $B_i$ of $\pi_1(\Sigma_i)$ which is not a boundary subgroup of $\Sigma_{i+1}$. We set $H_i =H_{\Sigma_i, B_i}$: by Example \ref{EmbeddingsHSigma}, it embeds in $H_{\Sigma_{i+1}, B_{i+1}}$. 
    By Proposition \ref{FreeProduct}, there exists elements $\tau_{i} \in H_{\Sigma_{i}}$ satisfying the required condition.
    \end{proof}
    
    We are now almost ready to prove Theorem \ref{MainResultOneFloor}. It will follow from 
    \begin{proposition} \label{OneExtraFloor} Let $T$ be a minimal rank hyperbolic tower over $\F(a)$, whose top floor $(T, T', \Lambda, r)$ has a unique surface $\Sigma$. Fix parameter set $A=T'$ in $T$, and denote by ${\cal G}$ the group $\Aut_{T'}(T)$.
    
    Let $s$ be any finite generating tuple of $T'$, extend it to a generating tuple $t,s$ for $T$. Then there is a witnessing sequence in ${\cal G}$ of length $\omega^{-\chi(\Sigma)}$ relative to $(t,s)$. 
    \end{proposition}

    \begin{proof}  Pick a vertex group $\pi_1(\Sigma)$ of $\Lambda$ corresponding to $\Sigma$ such that $\pi_1(\Sigma) \cap T'$ is a boundary subgroup $B$ of the surface group. 
    
    Let $H_{\Sigma}$ be the group of automorphisms of $\pi_1(\Sigma)$ fixing $B$ defined in \ref{HSigmaDef}. It embeds naturally in $\Aut_{T'}(T)$ (see Example \ref{EmbeddingsHSigma}(2)).
     We want to apply Corollary \ref{rank times omega iterated} to the sequence $H_0=\{1\} \leq H_1 \leq \ldots \leq H_{m}= H_{\Sigma}$ given by Lemma \ref{SequenceFreeProducts} above, with
     \begin{itemize}
     \item $A = T'$;
     \item $M=T$ - this is a hyperbolic tower, so it is a model of the theory of free groups;
     \item $v=(t,s)$ which lies in $T$;
     \end{itemize}
    The sets $H_{\Sigma_i} \cdot v$ are all definable over $v$ for $i<m$  by Proposition \ref{DefinabilityOfOrbits}. This proves the result.
    \end{proof} 
    
    We get the following 
    \begin{corollary} \label{OneFloorFiber} Let $T'$ be a hyperbolic tower over $\F(a)$. Suppose $T$ is a hyperbolic tower with one floor over $T'$ with a single surface. Let $(t,s,a)$ be any generating set for $T$ such that $(s,a)$ generates $T'$, and choose corresponding finite presentations $\langle t,s,a \mid \Sigma_{T}(t,s,a) \rangle$ and $\langle s,a \mid \Sigma_{T'}(s,a)\rangle$ for $T$ and $T'$ respectively. Denote by $V^T_{T'}$ the set defined by $\Sigma_T(y,x,a) \wedge x=s$. 
    
    Then $\SR_{T'}(V^T_{T'}) \geq \omega^{-\chi(\Sigma)}$.
    \end{corollary}

    \begin{proof}
    	By Proposition \ref{SuperstableDoesNotExpand}, if $T$ does not have minimal rank as a tower over $T'$, the result holds - we may thus assume that $T$ has minimal rank. Let $\langle t,a \mid \Sigma_T(t,a)\rangle$ be a finite presentation for $T$.
    	
    	The previous result gives that there is a witnessing sequence in ${\cal G} = \Aut_{T'}(T)$ of length $\omega^{-\chi(\Sigma)}$ relative to $(t,s,a)$.
    	
    	We then apply Lemma \ref{WitnessingSequence} with $X = \{(t,s,a)\}$ and $Y = V_{T'}^T$. Note that for any $h,g \in \Aut_{T'}(T)$, if $g(X) \cap h(X) \neq \emptyset$ we must have $g(t)=h(t)$ so $g=h$. We thus get $\SR_{T'}(V^T_{T'}) \geq \omega^{-\chi(\Sigma)}$. 
    \end{proof}
    
    Theorem \ref{MainResultOneFloor} follows immediately by applying the corollary with $T'=\F(a)$.
    
  \subsection{Proof of Theorem \ref{MainResultSeveralFloors}}

 To prove Theorem \ref{MainResultSeveralFloors}, we will proceed by induction on the number of floors.  Note that we may assume $T$ has minimal rank otherwise the result follows by \ref{SuperstableDoesNotExpand}.
 
  Recall that we want to show that if $T$ is a tower over $\F(a)$ with floor surfaces $\Sigma_1, \ldots, \Sigma_n$, the Shelah rank relative to  $\F(a)$ of $V^T$ is at least $\omega^{-\chi(\Sigma_n)} + \ldots + \omega^{-\chi(\Sigma_1)}$.
  
 We will in fact show by induction that if $T$ is an $n$ floor hyperbolic tower with floor surfaces $\Sigma_1, \ldots, \Sigma_n$ over a tower $T'$ over $\F(a)$, then $\SR_{T'}(V^T_{T'}) \geq \omega^{-\chi(\Sigma_n)} + \ldots + \omega^{-\chi(\Sigma_1)}$, where $V^T_{T'}$ is defined by $\Sigma_T(y,x,a)=1 \wedge x=s$ for some presentation $\langle t,s,a \mid \Sigma_T(t,s,a) \rangle$ of $T$ such that $s,a$ generates $T'$.
 
The case where $n=1$ is exactly Corollary \ref{OneFloorFiber}.
 
 Suppose we know that the result holds whenever $T$ has at most $n-1$ floors over $T'$. Let $T_0$ be a tower over $\F(a)$, and let $T_n$ be a minimal rank tower which is built over $T_0$ by adding $n$ floors, with corresponding surfaces $\Sigma_1, \ldots, \Sigma_n$. Choose a presentation $\langle s_n, \ldots, s_0, a \mid \Sigma_{T_n} (s_n, \ldots, s_0, a)\rangle$ of $T_n$ such that the intermediate towers have presentations of the form $\langle s_j, \ldots, s_0, a \mid \Sigma_{T_j} (s_j, \ldots, s_0, a)\rangle$.
 
Denote by $T_1$ the tower obtained from $T_0$ by adding the first floor (with surface $\Sigma_1$). By induction hypothesis we know that $\SR_{T_1}(V^{T_n}_{T_1}) \geq \omega^{-\chi(\Sigma_{n})} + \ldots + \omega^{-\chi(\Sigma_2)}$. Note that $V^{T_1}$ is a projection of $V^{T_n}$, that both are defined over $T_0$ only, and that the set $X_{(s_1, s_0, a)}$ defined by 
$$\phi(x_n, \ldots, x_1, x_0, a) := \Sigma_{T_n}(x_n, \ldots, x_1, x_0, a)=1 \wedge x_1 = s_1 \wedge x_0=s_0$$
(in other words the fiber of $V^{T_n}$ above $s_1 \in V_{T_1}$) is definably isomorphic over $T_0$ to the variety $V^{T_n}_{T_1}$. In particular it has Shelah rank relative to $T_1$ (and thus also relative to $T_0$) at least $\omega^{-\chi(\Sigma_{n})} + \ldots + \omega^{-\chi(\Sigma_2)}$.

Now we know by Proposition \ref{OneExtraFloor} that there is a witnessing sequence $\{{\cal S}_{\beta}\}_{\beta < \omega^{-\chi(\Sigma_1)}}$ in ${\cal G} = \Aut_{T_0}(T_1)$ relative to $(s_1, s_0, a)$ and for the parameter set $A = T_0$.

Now apply Lemma \ref{WitnessingSequence} with $X = X_{(s_1, s_0, a)}$ and $Y = V^{T_n}_{T_0}$. Note that for any $g, h \in {\cal G}$ we have that if $g(X) \cap h(X) \neq \emptyset$ then in fact $g(s_1)=h(s_1)$ so $g=h$, and $g(X)$ is defined by $\phi(x_n, \ldots, x_1, a) := \Sigma_{T_n}(x_n, \ldots, x_1, s_0, a)=1 \wedge x_1 = g(s_1)$ which is contained in $Y$. We get 
\begin{eqnarray*}
\SR_{T_0}(V_{T_0}^{T_n}) &\geq& \SR_{T_0}(X) + \omega^{-\chi(\Sigma_1)} \\
                         &\geq& \omega^{-\chi(\Sigma_{n})} + \ldots + \omega^{-\chi(\Sigma_2)} + \omega^{-\chi(\Sigma_1)}
\end{eqnarray*}
which finishes the proof.

\section{Finding a free product} \label{FreeProductSec}

  \subsection{Hyperbolic metric spaces}
    
    We will use standard properties and results about Gromov hyperbolic spaces and their isometries. We will give an overview of the properties that will be of use in the sequel, referring the reader to \cite{bridson2011metric} or \cite{ghys1990espaces} for a detailed account.
    
    For points $x, y, w$ of a metric space $(X, d)$, we define the Gromov product of $x, y$ relative to $w$ to be
    $$ (x,y)_{w} = \frac{d(x,w)+d(y,w)-d(x,y)}{2}$$
    
    We say that a metric space $(X,d)$ is hyperbolic (in the sense of Gromov) if there is some universal constant $\delta>0$  such that the following inequality holds for any $x,y,z,w\in X$
    \begin{align*}
    (x,z)_{w}\geq \min\{(x,y)_{w},(y,z)_{w}\}-\delta
    \end{align*}

    A metric space $(X,d)$ is geodesic if any two distinct points in $X$ can be joined by a geodesic segment. In particular, the geometric realization of a connected graph is a geodesic metric space.
    
    A useful property of geodesic hyperbolic spaces is the following
    \begin{lemma}\label{stability} Let $(X,d)$ be a geodesic hyperbolic metric space.
    There is a constant $Q$ such that for any geodesic segments $\sigma,\tau$ with endpoints $\{x_{0},x_{1}\}$ and $\{y_{0},y_{1}\}$ respectively, if $d(x_{i},y_{i})<L$ for $i\in\{0,1\}$, then $d_{H}(\sigma,\tau)<L+Q$.
    \end{lemma}

    An important feature of a Gromov-hyperbolic space $X$ is the possibility of defining a metrizable topological space called boundary of $X$ at infinity, or $\partial X$. 
    
    Points of $\partial X$ are defined as equivalence classes of sequences of points $(x_{n})_{n \in \N}\subset X$ for which the value $(x_{m},x_{n})_{z}$ tends to infinity with $m$ and $n$ for some (any) $z\in X$. One may think of any such sequences as converging to the corresponding point in the boundary in a certain sense.
    Sequences $(x_{n})_{n\in \N}$ and $(y_{n})_{n \in \N}$ are in the same class precisely when
    \begin{align*}
    \lim_{n,m\to\infty}\,(x_{n},y_{m})_{z}=\infty
    \end{align*}
    for some (all) $z\in X$. Letting the constant $M$
    range in $\R_{+}$ in the expression
    \begin{align*}
    \setof{[(y_{n})_{n}]\in\partial X}{(x_{n},y_{m})_{z}\geq M}
    \end{align*}
    yields a system of neighbourhoods of $[(x_{n})_{n}]$ in $\partial X$.
    
    Any isometry $f$ of $X$ induces a homeomorphism $\bar{f}$ of $\partial X$. Moreover we have    
    \begin{lemma} \label{IsometryBoundary}
    Given a fixed $x\in X$ one can metrize $\partial X$ in a way that $\bar{f}$ is an isometry of $\partial X$ for any isometry $f$ fixing $x$.
    \end{lemma}
    
    It is a classical result that any isometry $f$ of a hyperbolic space $(X,d)$ falls into exactly one of three following cases. If orbits of points of $X$ by $f$ are bounded we say $f$ is \textbf{elliptic}; if $\bar{f}$ fixes a unique point of $\partial X$, then we say $f$ is \textbf{parabolic}. 
    The remaining case is that in which $\bar{f}$ fixes exactly two points $e^+, e^-$ of $\partial X$ - in that case we say it is \textbf{hyperbolic}. When $f$ is hyperbolic the action of $\bar{f}$ on $\partial X$ exhibits a North-South dynamics: there are respective neighbourhoods $U^{+},U^{-}$ of each of the two points $e^{+},e^{-}$ such that for some $N\in\N$ any positive power of $\bar{f}^{N}$ sends $\partial\mathcal{C}\setminus U^{-}$ into $U^{+}$ and any negative one sends $\partial\mathcal{C}\setminus U^{+}$ into $U^{-}$.
    
    A geodesic ray in $X$ corresponds to a unique point at infinity to which
    all divergent sequences of points in it converge. Likewise, a bi-infinite geodesic in $X$ corresponds to two distinct points, or ends, in the boundary.
    Two geodesics determine the same pair of points if and only if they are at finite Hausdorff distance or, equivalently if for parametrizations $\tau$ and $\tau'$ of said geodesics
    the distance $d(\tau(t),\tau(t'))$ is bounded from above by certain universal constant $P$ which depends only on the hyperbolicity constant $\delta$.
    In that situation we say that the two geodesics are parallel.
    
    \begin{definition}
    An axis for an hyperbolic isometry $f$ of a hyperbolic metric space $X$ is a bi-infinite geodesic $A$ in $X$ such that $f$ acts almost as a translation on $X$, i.e. such that $f^{n}(A)$ and $A$ are parallel for any $n\in \Z$. 
    
    Equivalently, $A$ is a geodesic joining the two fixed points $e^+,e^-$ of $\bar{f}$ in $\partial X$.
    \end{definition}

    \begin{lemma}
    \leavevmode\label{spacing}
    Let $f$ be a hyperbolic isometry of $X$ which has an axis $A$. Given $x\in X$, if we let $x_{n}=f^{n}(x)$ and $y_{n}$ a point in $A$ at minimum distance from $x_{n}$, then
    \begin{itemize}
    	\item there is $D>0$ such that the family $\{y_{j}\}$ is $D$-dense in $A$.
    	\item for some $K,T>0$ and any $i,j\in\Z$ we have
    	\begin{align*}
    		T|i-j|-K\leq d(y_{i},y_{j}),d(x_{i},x_{j})\leq T|i-j|+K
    	\end{align*}
    \end{itemize}
    \end{lemma}
    
    In general hyperbolic isometries can only be guaranteed to admit a quasi-axis -something less restrictive- rather than an axis as described above. They do however in the context relevant here; it is not essential, but will allow to simplify our presentation. 
    
  \subsection{Free product}
    
    \newcommand{\Sa}[0]{\pi_{1}(\Sigma_{0})}
    \newcommand{\So}[0]{\pi_{1}(\Sigma_{1})}
    \newcommand{\St}[0]{\pi_{1}(\Sigma_{1})}
    \newcommand{\Si}[0]{\pi_{1}(\Sigma_{0})}
    
    For any surface with boundary $\Sigma$, denote by $Aut^{*}(\pi_{1}(\Sigma))$ be the subgroup of automorphisms of $\pi_{1}(\Sigma)$ that restrict to an inner automorphism on each boundary subgroup and by $Out^{*}(\pi_{1}(\Sigma))$ its image in $Out(\pi_{1}(\Sigma))$.  
    
   The goal of this section is to prove Proposition \ref{FreeProduct}. Recall we are given a surface $\Sigma_1$, and a subsurface $\Sigma_0$ of $\Sigma_1$ such that $\Sigma_1 \setminus \Sigma_0$ is connected. We also have an embedding of the group $H_{\Sigma_0}$ in $H_{\Sigma_1}$, where $H_{\Sigma_1}$ is the group of automorphisms of $\pi_1(\Sigma_1)$ restricting to the identity on some fixed boundary subgroup $B$ and to a conjugation on the others.
  
    The natural map $\phi:H_{\Sigma_i}\to Out^{*}(\pi_1(\Sigma_{i}))$ is easily seen to be onto. The image of $H_{\Sigma_{0}}$ in $Aut^{*}(\So)$ fixes the subgroup $\pi_1(\Sigma_1\setminus\Sigma_0)$, which has trivial centralizer, so $\phi$ is necessarily injective on $H_{\Sigma_{0}}$.
    
    To prove the proposition it is thus enough to prove the existence of a non-trivial element $\tau' \in Out^{*}(\pi_1(\Sigma_1))$ such that
    $\subg{\tau',\phi(H_{\Sigma_{0}})}$ is isomorphic to $\subg{\tau'}\frp\phi(H_{\Sigma_{0}})$.
    Indeed, in that case it suffices to take as $\tau$ any lift of $\tau'$ to $H_{\Sigma_1}$. 
    
  \subsection{The modular group and the complex of curves}

    Consider a general compact surface  $\Sigma$ with nonempty boundary. The modular group $Mod(\Sigma)$ of a surface $\Sigma$ with boundary is the quotient $\Hom^+(\Sigma, \partial \Sigma)/ \sim$ where
    \begin{itemize}
    \item $\Hom^+(\Sigma, \partial \Sigma)$ is the group of orientable homeomorphisms of $\Sigma$ fixing the boundary $\partial\Sigma$,
    \item two homeomorphism are related by $\sim$ if they are homotopic by a homotopy which fixes the boundary setwise.
    \end{itemize}
    (Properly speaking the object defined above is referred to as the modular group of the surface obtained by replacing each boundary component by a puncture.)

    Every element $[h]\in Mod(\Sigma)$ determines a unique element $[h_{*}]\in Out^{*}(\pi_{1}(\Sigma))$. The resulting correspondence can be shown to be a homomorphism. The Dehn-Nielsen-Baer theorem tells us this map is in fact an isomorphism (see \cite[Theorem 8.1]{farb2011primer}), that is, $\Out^*(\pi_1(\Sigma)) \simeq \Mod(\Sigma)$. 
    
    A simple closed curve in $\Sigma$ is called essential if it is two sided and it cannot be homotoped into a boundary component.
    
    Suppose now $\Sigma$ has $b$ boundary components and genus $g$ such that $3g+b\geq 4$. The curve graph of $\Sigma$, which we will denote by $\mathcal{C}=\mathcal{C}(\Sigma)$, has as vertex set the collection of homotopy classes of essential simple closed curves. An edge joins two distinct homotopy classes precisely when their geometric intersection number is minimal among all pairs of non-homotopic essential simple closed curves (in particular, if $3g+b\geq 5$, if the two classes can be realized disjointly). This graph is connected and we can regard it as a path connected metric space.
    
    There is an obvious action of the group of homeomorphisms of $\Sigma$ on $\mathcal{C}$, which is easily seen to factor through the quotient $Mod(\Sigma)$.
    
    Now let us go back to our pair of surfaces $\Sigma_{0}$ and $\Sigma_{1}$, and groups of automorphisms $H_{\Sigma_1}, H_{\Sigma_0}$.
     As noticed above, the group $H_{\Sigma_1}$ maps onto $\Out^*(\pi_1(\Sigma_1))$, equivalently onto $G:=Mod(\Sigma_{1})$. Notice that since $H_{\Sigma_{0}}$ fixes at least one element corresponding to an essential simple closed curve its image $G_{0} = \phi(H_{\Sigma_0})$ in $G$ has to fix at least one vertex of $\mathcal{C}$. 
    
    Elements of $Mod(\Sigma_1)$ that act hyperbolically on the space $\mathcal{C}$ are called pseudo-Anosov.
    They are known to admit a geodesic axis $A$ (see Proposition 7.6 in \cite{masur2000geometry}). 
    We will prove the following:
    \begin{proposition}
    \label{perturbation claim} For any pseudo-Anosov $\gamma\in G = Mod(\Sigma_1)$ with pseudo-axis $A$, there are neighbourhoods $U^{+}$ and $U^{-}$ of the ends $e^{+}$ and $e^{-}$ of $A$ such that for any $h\in G_{0}\setminus\{1\}$:
    \begin{align*}
    	\bar{h}\cdot\{e^{+},e^{-}\}\cap (U^{+}\cup U^{-})=\nil
    \end{align*}
    \end{proposition}
    We shall now show how to derive Proposition \ref{FreeProduct} from this claim and then devote the rest of the section to its proof.
    \begin{proof}(of Proposition \ref{FreeProduct})
    Take some pseudo-Anosov $\gamma\in G$ and let $A$ be a geodesic axis for $\gamma$. We will show that the subgroup of $G$ generated by $G_0$ and some power $\gamma^n$ of $\gamma$ is isomorphic to $G_0 * \langle \gamma^n \rangle$. 

By Lemma \ref{IsometryBoundary}, given any vertex $v \in \mathcal{C}$, there is a metric on $\partial \mathcal{C}$ such that any element of $G$ fixing $v$ acts isometrically on $\partial \mathcal{C}$. We choose such a metric for $v$ a vertex fixed by $G_0$. 

    This allows one to replace the claim made by Proposition \ref{perturbation claim} by the even stronger assertion that $\bar{h}(U^{+}\cup U^{-})\cap (U^{+}\cup U^{-})=\nil$ for any $h\in G_{0}\setminus\{1\}$. Indeed, one can assume that $U^{\pm}$ is a ball centered at $e^{\pm}$ to begin with and then
    replace it with a ball of half its radius.
    
    Consider the sets $\mathcal{A}=U^{+}\cup U^{-}$ and $\mathcal{B}=\partial\mathcal{C}\setminus\mathcal{A}$. Recall that $\gamma$ has a North-South dynamics on the boundary, that is, for some $N\in\N$ any $f\in\subg{\gamma^{N}}\setminus\{1\}$ satisfies
    $$	\bar{f}\mathcal{B}\subset \mathcal{A}$$
    On the other hand, we have seen that for any non trivial element $h$ of $G_0$ we have
    $$ \bar{h}\mathcal{A}\subset\mathcal{B} $$
    thus if we take $\tau=\gamma^{N}$ the result follows as a direct consequence of the ping-pong lemma.
    \end{proof}
    
    The main ingredient in the proof of \ref{perturbation claim} is the following result (see Proposition 11 from \cite{bestvina2010asymptotic}), which states that the action of $G$ on $\mathcal{C}$ satisfies the WPD property.
    
    \begin{proposition}
    \label{WPD property}Let $G$ be the modular group of a closed surface with punctures. For every element $\gamma\in G$ which acts hyperbolically on $\mathcal{C}$, any $x\in\mathcal{C}$ and every $C>0$ there is $N=N_{x}(C)>0$ such that the following set is finite:
    \begin{align*}
    	\mathcal{F}_{x}(C)= \setof{g\in G}{d(x,g\cdot x),d(\gamma^{N}\cdot x,g\gamma^{N}\cdot x)\leq C}
    \end{align*}
    \end{proposition}

    For the rest of the section let $\gamma$ be a fixed pseudo-Anosov element of $G$ and $A$ a quasi-axis of $\gamma$.
    The key observation is that the theorem above implies a certain dichotomy holds for any $g\in G$: either $g\cdot A$ is parallel to $A$, i.e. at a close Hausdorff distance from $A$, or $g$ substantially perturbs $A$ - this is the object of the following result.
    \begin{lemma}
    \label{dychotomy lemma} Let $A$ be a geodesic quasi-axis for the pseudo-Anosov element $\gamma\in G$. Then for any $C>0$ big enough there are $M=M(C,\gamma) > 0$ such that for any $g\in G$ either:
    \begin{enumerate}[i)]
    	\item For any two points $x,y\in A$ at distance at least $M$ the inequality $d(g\cdot z,A)>C$ holds for at least one of the two points $x,y$.
    	\item The element $g$ preserves the ends of $A$ in $\partial\mathcal{C}$ set-wise. In that case  if $g$ fixes both ends of $A$ then there is $e\in \N$ such that $g\gamma^{e}g^{-1}=\gamma^e$, and  if it swaps them $g \gamma^{e} g^{-1}=\gamma^{-e}$.
    \end{enumerate}
    \end{lemma}

    \begin{proof}     Fix $x_{0}\in A\cap\mathcal{C}^{0}$. The points $x_{j}=\gamma^{j}\cdot x_{0}$ for $j\in\mathbb{Z}$ are at bounded distance from $A$, and the map $j \mapsto x_j$ is a quasi isometry. Thus there exists a constant $D$ such that any point on $A$ is $D$-close to one of the points $x_j$.
    
    Let $u$ and $v$ be points on $A$ which we think of as "far apart". Let $k \in \Z, M>0$ be such that $x_{k-M}$ is $D$-close to $u$ while $x_{k+M}$ is $D$-close to $v$. Note that $M$ goes to infinity as $d(u,v)$ goes to infinity.
    
   Suppose we are not in the first case, i.e. that we have $d(g \cdot u,A), d(g\cdot v, A) \leq C$. The image by $g$ of the geodesic segment of $A$ between $u$ and $v$ is at distance $C + 2 \delta$ of $A$, hence using Lemma \ref{spacing} we see that there is an index $l$ and a constant $Q$ which depends only on $\delta, C, D$ and $\gamma$ (and not on $u$ and $v$ or $g$) such that one of the two following possibilities holds
   \begin{enumerate}
   \item (orientation preserving) for each $j \in [-M, M]$ we have $d(g \cdot x_{k+j}, x_{l+j}) \leq Q$;
   \item (orientation reversing) for each $j \in [-M, M]$ we have $d(g \cdot x_{k+j}, x_{l-j}) \leq Q$.   
   \end{enumerate}
   
   Suppose we are in the first case, and pick some $N \leq M$: any element of the form $g_j =\gamma^{-l-j} g \gamma^{k+j}$ for $j \in [-M+N, M-N] \}$ satisfies 
   \begin{enumerate}
   \item[(i)] $d( g_j\cdot x_0, x_0) = d( \gamma^{-l-j} g \gamma^{k+j} \cdot x_0, x_0) =d( g \gamma^{k+j} \cdot x_0, \gamma^{l+j} \cdot x_0) = d( g \cdot x_{k+j}, x_{l+j})  \leq Q$;
   \item[(ii)] $d( g_j x_N, x_N)=d( \gamma^{-l-j} g \gamma^{k+j} x_N, x_N) = d( g \gamma^{k+j} x_N, \gamma^{l+j} x_N) = d( g \gamma^{k+(j+N)} \cdot x_0, \gamma^{l+(j+N)} \cdot x_0)\leq Q$
   \end{enumerate}
   
   If we are in the second case, note that for each $j \in [-M, M]$, applying $g \gamma^{j}$ to $x_{k}$ gives a point which is $Q$-close to $x_{l-j} = \gamma^{-j} \cdot x_l$, which is itself $Q$-close to $\gamma^{-j} g \cdot x_k$. Similarly, for any $N < M$ and $j \in [N-M, M-N]$, the points obtained by applying $g \gamma^j$ and $\gamma^{-j} g$ to $x_{k+N}$ are at distance at most $2Q$.
  Thus any element of the form $g_j = \gamma^{-k}g^{-1} \gamma^j g \gamma^{j+k}$ for $j \in [-M+N, M-N] \}$ satisfies 
   \begin{enumerate}
      \item[(i)] $d( g_j\cdot x_0, x_0) = d(\gamma^{-k}g^{-1} \gamma^j g \gamma^{j+k}\cdot x_0, x_0) = d(g \gamma^{j}\cdot x_k, \gamma^{-j} g \cdot x_k) \leq 2Q$;
      \item[(ii)] $d( g_j \cdot x_N, x_N) = d(\gamma^{-k}g^{-1} \gamma^j g \gamma^{j+k} \cdot x_N, x_N) = d(g \gamma^{j}\cdot x_{k+N}, \gamma^{-j} g\cdot x_{k+N}) \leq 2Q$;.
   \end{enumerate}
   
   Now suppose that $d(u, v)$ is large enough so that $M > \abs{{\cal F}_{x_0}(Q)} + N_{x_0}(Q)$ and $M > \abs{{\cal F}_{x_0}(2Q)} + N_{x_0}(2Q)$. 
   
   Then if we are in the first case, and since the elements $g_j$ for $j \in [-M+N_{x_0}(Q), M-N_{x_0}(Q)]$ are all in ${\cal F}_{x_0}(Q)$, by the pigeon hole principle, there must exist indices $i$ and $j$ such that $g_i =g_j$, that is $\gamma^{-l-j}g \gamma^{k+j}=\gamma^{-l-i}g \gamma^{k+i}$, which implies that $g\gamma^{i-j}g^{-1}=\gamma^{i-j}$. In particular $g$ fixes the two ends of $A$.
   
   If we are in the second case, since the elements $g_j$ for $j \in [-M+N_{x_0}(2Q), M-N_{x_0}(2Q)]$ are all in ${\cal F}_{x_0}(2Q)$, by the pigeon hole principle, there must exist indices $i$ and $j$ such that $g_i =g_j$, that is $\gamma^{-k}g^{-1} \gamma^j g \gamma^{j+k} = \gamma^{-k}g^{-1} \gamma^i g \gamma^{i+k}$ which implies that $g \gamma^{j-i} g^{-1} =\gamma^{i-j}$. Thus $g$ swaps the two ends of $A$.
   \end{proof}

    Let $\mathcal{G}$ be the subgroup of elements of $G$ which preserve the ends of $A$ and $\mathcal{G}^{+}$ that of those which fix them point-wise. Of course, $\mathcal{G}^{+}$ is a subgroup of index $\leq 2$ in $\mathcal{G}$.
    
    \begin{lemma}
    \begin{align*}
    	\mathcal{G}\cap G_{0}=\{1\}
    \end{align*}
    \end{lemma}
    \begin{proof}
    Since we know that the modular group of surfaces with boundary is torsion free (see \cite[Corollary 7.3, p.201]{farb2011primer}), all we need to show is that $\mathcal{G}^{+}\cap G_{0}$ is finite.
    
    Notice that given a parametrization $\tau:\R\to X$ of $A$, we know that there is a constant $P>0$ such that for any $g\in\mathcal{G}^{+}$ there exists $T_g\in\R$ such that $d(g \cdot \tau(t), \tau(t+T_g))\leq P$ for any $t\in \R$.
    
    Now if $g \in G_0$ we have $d(\tau(0), g\cdot \tau(0)) \leq 2d(v, \tau(0))$ since $g \cdot v = v$. In particular $\abs{T_g} = d(\tau(0), \tau(T_g)) < 2d(v, \tau(0)) +P$. 
    Thus for any $g\in G_{0}\cap\mathcal{G}^{+}$ and $x = \tau(t) \in A$ we have 
    \begin{align*}
    d(x, g \cdot x) &= d(\tau(t), g \cdot \tau(t)) \\
                           &\leq d(\tau(t), \tau(t+T_g)) + d(\tau(t+T_g), g \cdot \tau(t)) \\
                           &\leq \abs{T_g}+P \leq 2d(v, \tau(0)) +2P.
    \end{align*}
    That is, the distance $d(x,g \cdot x)$ is bounded uniformly in $g \in G_{0}\cap\mathcal{G}^{+}$ and $x\in A$, which implies $d(\gamma^{m}x,g\gamma^{m}x)$ is also bounded uniformly in $g\in G_{0}\cap\mathcal{G}^{+}$ and $m\in\Z$ so that an application of Proposition \ref{WPD property} above finishes the proof.
    \end{proof}
    
    We are now in a condition to prove claim \ref{perturbation claim}:
    \begin{proof}
    Let $v$ be a vertex of $\mathcal{C}$ fixed by $G_{0}$ and $w$ one in $A$. Let $D=d(v, w)$.
    Denote by $\rho^{+}$ and $\rho^{-}$ the two geodesic rays in $\mathcal{C}$ starting at $v$, corresponding to the ends $e^{+}$ and $e^{-}$ in $\partial \mathcal{C}$ respectively.
    Let now $M:=M(3\delta,\gamma)$ be the constant provided by \ref{dychotomy lemma}.
    \newcommand{\M}[0]{M+D+2\delta}
    
    In the particular context of the graph $\mathcal{C}$ all points in the boundary can be seen as equivalence classes of geodesic rays. The following sets define neighbourhoods around $e^{+}=[\rho^{+}]$ and $e^{-}=[\rho^{-}]$ respectively.
    \begin{align*}
    	V^{+}=\{[\sigma]\,|\,\exists N\,\forall t\geq N\,\,(\sigma(t),\rho^{+}(t))_{v}\geq\M\}
    	\\ V^{-}=\{[\sigma]\,|\,\exists N\,\forall t\geq N\,\,(\sigma(t),\rho^{-}(t))_{v}\geq\M\}
    \end{align*}
    Now pick any $h\in G_{0}\setminus\{1\}$. We claim that $[h\cdot \rho^{+}]\notin V^{+}$. Indeed, suppose this was not the case.
    Then there are points $x\in\rho^{+}$ and $x'\in h\cdot\rho^{+}$ such that $(x,x')_{v}\geq\M$. Take geodesic segments
    $\theta$ and $\theta'$ from $v$ to $x$ and $x'$ respectively (they are initial segments of $\rho^+$ and $h \cdot \rho^+$ respectively).
    
    The assumption that $(x,x')_{v}\geq\M$ implies that points in $\theta$ and $\theta'$ at the same distance $\leq\M$ from $v$ are $\delta$-close to each other. 
    On the other hand, the complement in $\theta$ of its initial segment of length $D$ is contained in a $\delta$ neighbourhood of $A$. Hence the image by $h$ of some segment of length at least $M$ in $A$ is contained in a $3\delta$-neighbourhood of $A$, contradicting the choice of $M$. Using a similar argument for the remaining cases, we complete the proof of the fact that $[h\cdot\rho^{\pm}]\notin V^{\pm}$.
    \end{proof}
    
\bibliographystyle{plain}
\providecommand{\bysame}{\leavevmode\hbox to3em{\hrulefill}\thinspace}
\providecommand{\MR}{\relax\ifhmode\unskip\space\fi MR }
\providecommand{\MRhref}[2]{%
  \href{http://www.ams.org/mathscinet-getitem?mr=#1}{#2}
}
\providecommand{\href}[2]{#2}

\end{document}